\documentclass[12pt,reqno]{amsart}
\usepackage{amsmath, amsfonts, amssymb, cite, amsthm}
\def\Z{\mathbb Z}

\def\1{{\bf 1}}

\def\pmod #1{\ ({\rm{mod}}\ #1)}

\theoremstyle{plain}
\newtheorem{theorem}{Theorem}

\newtheorem{lemma}{Lemma}

\theoremstyle{definition}
\newtheorem*{acknowledgment}{Acknowledgments}
\theoremstyle{remark}

\begin{document}

\title[A generalization of Rodriguez-Villegas' supercongruences]{A $q$-Dwork-type generalization of Rodriguez-Villegas' supercongruences}

\begin{abstract}
Guo and Zudilin [Adv. Math. 346 (2019),
329--358] developed an analytical method, called `creative
microscoping', to prove many supercongruences by establishing their
$q$-analogues. In this paper, we apply this method to give a
$q$-Dwork-type generalization of Rodriguez-Villegas' supercongruences,
which was recently conjectured by Guo and Zudilin.
\end{abstract}
\author{He-Xia Ni}
\address{Department of Applied Mathematics, Nanjing Audit University\\Nanjing 211815,
People's Republic of China}
\email{nihexia@yeah.net}

\keywords{congruence; $q$-supercongruences, $q$-Pochhammer symbol, cyclotomic polynomial}
\thanks{The author gratefully acknowledges the support from the Natural Science Foundation of the Higher Education Institutions of Jiangsu Province (20KJB110023) and National Natural Science Foundation of China (12001279).  }
\subjclass[2010]{Primary 11B65; Secondary 05A10, 05A30, 11A07}
\maketitle

\section{Introduction}
\setcounter{theorem}{0}\setcounter{lemma}{0}\setcounter{equation}{0}

Let $p>3$ be a prime and $\big(\frac{\cdot}{p}\big)$ be the Legendre symbol modulo $p$. In 2003, E. Morterson \cite{M03,M04} proved  the following  supercongruences involving hypergeometric functions and Calabi-Yau manifolds,
\begin{align}
\sum_{k=0}^{p-1}\frac{\binom{2k}{k}^2}{16^k}\equiv&\left(\frac{-1}{p}\right)\pmod{p^2},
\label{1mor1} \\[5pt]
\sum_{k=0}^{p-1}\frac{\binom{3k}{k}\binom{2k}{k}}{27^k}\equiv&\left(\frac{p}{3}\right)\pmod{p^2},
\label{mor1} \\[5pt]
\sum_{k=0}^{p-1}\frac{\binom{4k}{2k}\binom{2k}{k}}{64^k}\equiv&\left(\frac{-2}{p}\right)\pmod{p^2},
\label{mor2}  \\[5pt]
\sum_{k=0}^{p-1}\frac{\binom{6k}{3k}\binom{3k}{k}}{432^k}\equiv&\left(\frac{-1}{p}\right)\pmod{p^2},
\label{mor3}
\end{align}
which were first conjectured by F. Rodriguez-Villegas \cite{RV}. In 2014, Z.-H. Sun \cite{ZHSun14} gave an elementary proof of \eqref{1mor1}--\eqref{mor3}
by showing that $$\sum_{k=0}^{p-1}\binom{-x}{k}\binom{x-1}{k}\equiv (-1)^{\langle -x\rangle_p}\pmod{p^2}$$ for any $p$-adic integer $x$,
where $\langle x\rangle_m$ denotes the least nonnegative residue of $x$ modulo $m$. In 2017, J.-C. Liu \cite{JCL} stated that for any $x\in \{1/2,1/3,1/4,1/6\}$ and any positive integer $n$,
$$\sum_{k=0}^{pn-1}\binom{-x}{k}\binom{x-1}{k}\equiv (-1)^{\langle-x\rangle_p}\sum_{k=0}^{n-1}\binom{-x}{k}\binom{x-1}{k}\pmod{p^2}.$$

In recent years, $q$-analogues of  supercongruences were widely investigated, and
a variety of techniques, such as  asymptotic estimate, basic hypergeometric transformation, creative microscoping, $q$-WZ pair and $q$-Zeilberger algorithm etc., were involved.
For example, in  \cite{GZ1}, Guo and Zudilin introduced a new method called creative microscoping,
and used this method to proved several new Ramanujan-type $q$-congruences in a unifed way.
For more related results and the latest progress, the reader is referred to \cite{G2,G3,G1,GS,GZ14,GZ1,GZ2,LW,NP2, WY1,WY2,WY3,Zudilin2}.

In 1969, Dwork \cite{BD} studied a question of continuing analytical solutions $f(z)=\sum_{k=0}^{\infty}A_kz^k$ of linear differential equations via $p$-adic analysis.
 A general strategy was to prove that the truncated sums $f_r(z)=\sum_{k=0}^{p^r-1}A_kz^k,$ where $r=0,1,2,\ldots,$ satisfy the so-called Dwork congruences \cite{MV}
 \begin{align}\label{DWork}
 \frac{f_{r+1}(z)}{f_r(z^p)}\equiv\frac{f_r(z)}{f_{r-1}(z^p)}\pmod{p^r\Z_{p}[[z]]}\quad\text{for $r=1,2,\ldots$}.
 \end{align}
 Moreover, for some $m\geq 2,$ provided the congruences \eqref{DWork} hold modulo a higher power of $p$, such as, $$\frac{f_{r+1}(z)}{f_r(z^p)}\equiv\frac{f_r(z)}{f_{r-1}(z^p)}\pmod{p^{mr}\Z_{p}[[z]]}\quad\text{for $r=1,2,\ldots$}.$$  We refer to this type of congruences  as Dwork-type supercongruences. Recently, Guo and Zudilin \cite{GZ2} proved some Dwork-type supercongruences by establishing their $q$-ananlogues. For example, they proved that,
for any odd positive integer $n>1$  and integer $r\geq1$, modulo $\prod_{j=1}^{r}\Phi_{n^j}(q)^2$,
\begin{align}\label{more4}
\sum_{k=0}^{(n^r-1)/d}\frac{2(q;q^2)_k^2q^{2k}}{(q^2;q^2)_k^2(1+q^{2k})}\equiv\left(\frac{-1}{n}\right)\sum_{k=0}^{(n^{r-1}-1)/d}\frac{2(q^n;q^{2n})_k^2q^{2nk}}{(q^{2n};q^{2n})_k^2(1+q^{2nk})},
\end{align}
where $d=1,2.$  Here and in what follows, the {\it $q$-shifted factorial} \cite{GR} is defined by
$$
(x;q)_n=\begin{cases}
(1-x)(1-xq)\cdots(1-xq^{n-1}) &\text{if }n\geq 1,\\
1 &\text{if }n=0,
\end{cases}
$$
and the {\it $n$-th cyclotomic polynomial} is defined as
$$
\Phi_n(q):=\prod_{\substack{1\leq k\leq n\\ (n,k)=1}}(q-e^{2\pi\sqrt{-1}\cdot\frac{k}{n}}).
$$
Moreover, for polynomials $A_1(q),A_2(q),P(q)\in \Z[q],$ we say that $A_1(q)/A_2(q)\equiv 0\pmod{P(q)}$ if $P(q)$ divides $A_1(q)$ and is relatively prime to $A_2(q).$
More generally, if the numerator of the reduced form of the difference between two rational functions $B(q)$ and $C(q)$ is divisible by  $P(q),$ we say that $B(q)\equiv C(q)\pmod{P(q)}.$

Motivated by Guo and Zudilin's work \cite{GZ2}, we shall give the
following result, which was originally conjectured by Guo and Zudilin \cite[Conjecture 3.13]{GZ2}.

\begin{theorem}\label{Th1}
Let $m$ and $s$ be positive integers with $s < m$. Let $n>1$ be an odd integer with $n \equiv \pm1\pmod{m}$. Then, for $r \geq 2$, modulo $\prod_{j=1}^r \Phi_{n^j}(q)^2$,
\begin{align}\label{DT1}
\sum_{k=0}^{n^r-1}\frac{2(q^s;q^m)_k(q^{m-s};q^m)_kq^{mk}}{(q^m;q^m)_k^2(1+q^{mk})}
\equiv(-1)^{\langle-s/m\rangle_n}\sum_{k=0}^{n^{r-1}-1}\frac{2(q^{sn};q^{mn})_k(q^{mn-sn};q^{mn})_kq^{mnk}}{(q^{mn};q^{mn})_k^2(1+q^{mnk})}.
\end{align}
\end{theorem}

Clearly, the $q$-supercongruence \eqref{more4} is just the $(m,s)=(2,1)$ case of \eqref{DT1}.

Note that
\begin{equation*}
\Phi_d(1)=\begin{cases}p &\text{if }d=p^k\text{ for some prime }p,\\
1 &\text{otherwise}.
\end{cases}
\end{equation*}
Letting $n=p$ be a prime and $q\rightarrow 1$ in \eqref{DT1}, we obtain
\begin{align*}
\left(\sum_{k=0}^{p^r-1}\binom{-s/m}{k}\binom{-(m-s)/m}{k}-(-1)^{\langle-s/m\rangle_p}\sum_{k=0}^{p^{r-1}-1}\binom{-s/m}{k}\binom{-(m-s)/m}{k}\right)\bigg/p^{2r}
\end{align*}
is a $p$-adic integer.
Moreover, since $\gcd(m,p)=1$ and $1\bigg/\left(\binom{-s/m}{p^{r-1}}\binom{-(m-s)/m}{p^{r-1}}\right)\in \Z_{p}$, the number
$$W_{p,p^{r-1}}=\frac{\sum_{k=0}^{p^r-1}\binom{-s/m}{k}\binom{-(m-s)/m}{k}-(-1)^{\langle-s/m\rangle_p}\sum_{k=0}^{p^{r-1}-1}\binom{-s/m}{k}\binom{-(m-s)/m}{k}}{p^{2r}\binom{-s/m}{p^{r-1}}\binom{-(m-s)/m}{p^{r-1}}}$$
is a $p$-adic integer. This partially confirms the  $n=p^{r-1}$ case of a conjecture of Z.-W. Sun \cite[Conjecture 10]{ZWSun19}.
On the other hand, the $n=p$ and $q\to 1$ case of \eqref{DT1} with $m=3,4,6$ also confirms some predictions of Roberts and Rodriguez-Villegas from \cite{RR}.

The rest of the paper is arranged as follows. The proof of Theorem \ref{Th1} will be given in Section 2
using the creative microscoping method developed by Guo and Zudilin \cite{GZ1}.
More precisely, to prove Theorem \ref{Th1}, we shall prove its generalization with an extra parameter $a$ so that the corresponding congruence holds modulo $\prod_{j=0}^{n^{r-1}-1}(1-aq^{n(mj+s)})\prod_{j=0}^{n^{r-1}-1}(a-q^{n(mj+m-s)}).$ Since the polynomials $1-aq^{n(mj+s)}$ and $a-q^{n(mj+m-s)}$ are pairwise relatively prime for any $j$ with $0\leq j\leq n^{r-1}-1,$ this generalized $q$-congruence can be established modulo these polynomials individually. Finally, by
taking the limit $a\rightarrow1$, we obtain the desired $q$-supercongruence in Theorem \ref{Th1}.

\section{Proof of Theorem \ref{Th1}}
\setcounter{lemma}{0}
\setcounter{theorem}{0}
\setcounter{corollary}{0}
\setcounter{remark}{0}
\setcounter{equation}{0}
\setcounter{conjecture}{0}

We need the following lemma, which was proved by Guo \cite[Corollary 1.4]{G2}.
\begin{lemma}\label{Th1proofLemma1}
Let $m,n$ and $s$ be positive integers with $\gcd(m,n)=1$ and $n$ odd. Then, modulo $(1-aq^{s+m\langle-s/m\rangle_n})(a-q^{m-s+m\langle(s-m)/m\rangle_n}),$
\begin{align}\label{Th1proofDT1}
\sum_{k=0}^{n-1}\frac{2(aq^s;q^m)_k(q^{m-s}/a;q^m)_kq^{mk}}{(q^m;q^m)_k^2(1+q^{mk})}\equiv(-1)^{\langle-s/m\rangle_n}.
\end{align}
\end{lemma}

In order to prove Theorem \ref{Th1}, we need to establish the following two parametric
generalizations.
\begin{theorem}\label{Th1proof}
Let $m$, $n$ and $s$ be positive integers with $s < m$, $n \equiv 1\pmod m$ and $n$ odd. Let $r \geq 2$ be an integer and $a$ an indeterminate. Then, modulo
\begin{align}\label{divide1}
\prod_{j=0}^{n^{r-1}-1}(1-aq^{n(mj+s)})(a-q^{n(mj+m-s)}),
\end{align}
we have
\begin{align}\label{DT3}
&\sum_{k=0}^{n^r-1}\frac{2(aq^s;q^m)_k(q^{m-s}/a;q^m)_kq^{mk}}{(q^m;q^m)_k^2(1+q^{mk})} \notag\\[5pt]
&\quad\equiv(-1)^{\langle-s/m\rangle_n}\sum_{k=0}^{n^{r-1}-1}\frac{2(aq^{sn};q^{mn})_k(q^{mn-sn}/a;q^{mn})_kq^{mnk}}{(q^{mn};q^{mn})_k^2(1+q^{mnk})}.
\end{align}
\end{theorem}

\begin{proof} It suffices to show that both sides of \eqref{DT3} are identical when we take $a=q^{-n(mj+s)}$  for any $j$ with $0\leq j\leq n^{r-1}-1$, i.e.,
\begin{align}\label{identity1}
&\sum_{k=0}^{n^r-1}\frac{2(q^{s-n(mj+s)};q^m)_k(q^{m-s+n(mj+s)};q^m)_kq^{mk}}{(q^m;q^m)_k^2(1+q^{mk})}\notag\\
&\quad=(-1)^{\langle-s/m\rangle_n}\sum_{k=0}^{n^{r-1}-1}\frac{2(q^{-nmj};q^{mn})_k(q^{mn+nmj};q^{mn})_kq^{mnk}}{(q^{mn};q^{mn})_k^2(1+q^{mnk})},
\end{align}
or $a=q^{n(mj+m-s)}$ for any $j$ with $0\leq j\leq n^{r-1}-1$, i.e.,
\begin{align}\label{identity2}
&\sum_{k=0}^{n^r-1}\frac{2(q^{s+n(mj+m-s)};q^m)_k(q^{m-s-n(mj+m-s)};q^m)_kq^{mk}}{(q^m;q^m)_k^2(1+q^{mk})}\notag\\
&\quad=(-1)^{\langle-s/m\rangle_n}\sum_{k=0}^{n^{r-1}-1}\frac{2(q^{mjn+mn};q^{mn})_k(q^{-mjn};q^{mn})_kq^{mnk}}{(q^{mn};q^{mn})_k^2(1+q^{mnk})}.
\end{align}
It is easy to see that $n^{r}-1\geq nj+s(n-1)/m$ for $0\leq j\leq n^{r-1}-1$, and $n^{r}-1\geq nj+(m-s)(n-1)/m$ for $0\leq j\leq n^{r-1}-1$.
Since $n \equiv 1\pmod m$, we know $\langle-s/m\rangle_n=s(n-1)/m$ and $\langle (s-m)/m\rangle_n=(m-s)(n-1)/m$. By Lemma \ref{Th1proofLemma1} the left-hand side of \eqref{DT3} is equal to
$$(-1)^{nj+s(n-1)/m}.$$ Likewise, the right-hand side of \eqref{DT3} is equal to $$(-1)^{s(n-1)/m}(-1)^{(mj+s-s)/m}=(-1)^{nj+s(n-1)/m}.$$ This proves (\ref{identity1}).
Similarly, we can also prove the identity \eqref{identity2} is true. Namely, the $q$-congruence \eqref{DT3} is true modulo \eqref{divide1}.
\end{proof}
\begin{theorem}\label{Th2proof}
Let $m$, $n$ and $s$ be positive integers with $s < m$, $n \equiv -1\pmod m$ and $n$ odd. Then, for $r \geq 2$, modulo
\eqref{divide1}
we have
\begin{align}\label{DT4}
&\sum_{k=0}^{n^r-1}\frac{2(q^s/a;q^m)_k(aq^{m-s};q^m)_kq^{mk}}{(q^m;q^m)_k^2(1+q^{mk})} \notag\\[5pt]
&\quad\equiv(-1)^{\langle-s/m\rangle_n}\sum_{k=0}^{n^{r-1}-1}\frac{2(aq^{sn};q^{mn})_k(q^{mn-sn}/a;q^{mn})_kq^{mnk}}{(q^{mn};q^{mn})_k^2(1+q^{mnk})}.
\end{align}
\end{theorem}
\begin{proof} For $a=q^{-n(mj+s)}$ with $0\leq j\leq n^{r-1}-1$, by Lemma \ref{Th1proofLemma1}, the left-hand side of (\ref{DT4}) is equal to
\begin{align*}
&\sum_{k=0}^{n^r-1}\frac{2(q^{s+n(mj+s)};q^m)_k(q^{m-s-n(mj+s)};q^m)_kq^{mk}}{(q^m;q^m)_k^2(1+q^{mk})}\notag\\
&\quad=(-1)^{nj+s(n+1)/m-1}=(-1)^{j-1+s(n+1)/m},
\end{align*}
Similarly, the right-hand side of \eqref{DT4} is equal to
$$ (-1)^{\langle-s/m\rangle_{n}}(-1)^j=(-1)^{n-(n+1)s/m+j}=(-1)^{j-1+s(n+1)/m},$$
where we use the fact that $\langle-s/m\rangle_n=n-s(n+1)/m$ since $n\equiv-1\pmod{m}.$ And so the $q$-congruence \eqref{DT4} is true modulo $\prod_{j=0}^{n^{r-1}-1}(1-aq^{n(mj+s)})$.

For $a=q^{n(mj+m-s)}$ with $0\leq j\leq n^{r-1}-1$,
the left-hand side of \eqref{DT4} is equal to
\begin{align*}
&\sum_{k=0}^{n^r-1}\frac{2(q^{s-n(mj+m-s)};q^m)_k(q^{(m-s)+n(mj+m-s)};q^m)_kq^{mk}}{(q^m;q^m)_k^2(1+q^{mk})}\notag\\
&\quad=(-1)^{nj+n-s(n+1)/m}=(-1)^{j-1+s(n+1)/m},
\end{align*}
which is the same as the right-hand side of \eqref{DT4}. This proves \eqref{DT4}  modulo $\prod_{j=0}^{n^{r-1}-1}(a-q^{n(mj+m-s)})$.
\end{proof}
\begin{proof}[Proof of  Theorem \ref{Th1}] The limit of \eqref{divide1} as $a\rightarrow1$ has the factor $$\prod_{j=1}^r\Phi_{n^j}(q)^{2n^{r-j}},$$
where we use the fact that the sets
\begin{align*}
&\begin{cases}\{n(mj+s):j=0,\ldots, n^{r-1}-1 \},&\text{ }\\
\{n(mj+m-s):j=0,\ldots, n^{r-1}-1 \},\\
\end{cases}
\end{align*}
in total contain exactly $2n^{r-j}$ multiples of $n^j$ for $j=1,\ldots, r.$

On the other hand, the least common denominator of both sides of \eqref{DT3} is at most equal to
$
(q^m;q^m)^2_{n^r-1}\prod_{k=1}^{n^r-1}(1+q^{mk})
$
and its factor related to $\Phi_{n}(q),\Phi_{n^2}(q),\ldots$ is just
$$\prod_{j=1}^r\Phi_{n^j}(q)^{2\left\lfloor (n^{r}-1)/n^j\right\rfloor}=\prod_{j=1}^r\Phi_{n^j}(q)^{2(n^{r-j}-1)}.$$

Hence, letting $a\rightarrow 1$ in \eqref{DT3} we conclude that the $n\equiv1\pmod{m}$ case of \eqref{DT1} is true modulo $\prod_{j=1}^{r}\Phi_{n^j}(q)^2.$

Similarly, letting $a\rightarrow1$ in \eqref{DT4}, we deduce
that  the $n\equiv-1\pmod{m}$ case of \eqref{DT1} is also true. Namely, the $q$-congruence \eqref{DT1} is true modulo $\prod_{j=1}^r \Phi_{n^j}(q)^2$.
\end{proof}

\begin{acknowledgment}
%The authors are grateful to the anonymous referee for valuable comments that helped to improve the quality of the article.
The author thanks  V.J.W. Guo and C. Wang for their very helpful discussions on this paper.
\end{acknowledgment}

\end{document}